\documentclass[11pt]{article}
\usepackage{a4}
\usepackage{amsmath}
\usepackage{amssymb}
\usepackage{amsfonts}
\usepackage{amsthm}

\usepackage{hyperref}
\hypersetup{
    pdftitle={Determining Aschbacher classes using characters},
    pdfauthor={Sebastian Jambor},
    colorlinks=false,
    pdfborder={0 0 0}
}

\newtheorem{thm}{Theorem}[section]
\newtheorem{lemma}[thm]{Lemma}
\newtheorem{prop}[thm]{Proposition}

\theoremstyle{definition}
\newtheorem{definition}[thm]{Definition}

\evensidemargin\oddsidemargin
\DeclareMathOperator{\Ln}{L}
\DeclareMathOperator{\Uni}{U}

\DeclareMathOperator{\GL}{GL}
\DeclareMathOperator{\SL}{SL}

\DeclareMathOperator{\AGL}{AGL}

\DeclareMathOperator{\PSL}{PSL}
\DeclareMathOperator{\PSU}{PSU}

\DeclareMathOperator{\Gal}{Gal}
\DeclareMathOperator{\Cyc}{C}
\DeclareMathOperator{\Sym}{S}
\DeclareMathOperator{\Stab}{Stab}
\DeclareMathOperator{\Alt}{A}
\DeclareMathOperator{\Aut}{Aut}

\DeclareMathOperator{\Nor}{N}

\DeclareMathOperator{\trs}{tr}
\DeclareMathOperator{\rad}{rad}
\DeclareMathOperator{\car}{char}

\DeclareMathOperator{\im}{im}
\DeclareMathOperator{\tr}{tr}
\DeclareMathOperator{\Hom}{Hom}

\newcommand{\ol}[1]{\overline{#1}}

\newcommand{\C}{{\mathbb{C}}}

\newcommand{\F}{{\mathbb{F}}}

\newcommand{\mcC}{{\mathcal{C}}}
\newcommand{\tri}{\trianglelefteq}

\newcommand{\bgm}{\left(\begin{matrix}}
\newcommand{\enm}{\end{matrix}\right)}

\title{Determining Aschbacher classes using characters}
\author{Sebastian Jambor}
\date{}

\begin{document}
\maketitle

\begin{abstract}
\noindent
\textbf{Abstract.} Let $\Delta\colon G \to \GL(n, K)$ be an absolutely irreducible representation
of an arbitrary group~$G$ over an arbitrary field~$K$;
let $\chi\colon G \to K \colon g \mapsto \tr(\Delta(g))$ be its character.
In this paper, we assume knowledge of $\chi$ only, and study which properties of $\Delta$ can be inferred.
We prove criteria to decide whether $\Delta$ preserves a form, is realizable over a subfield, or
acts imprimitively on~$K^{n \times 1}$.
If $K$ is finite, this allows us to decide whether the image of $\Delta$ belongs to certain
Aschbacher classes.


\end{abstract}

\let\thefootnote\relax\footnote{The author was supported by the Alexander von Humboldt Foundation via a Feodor Lynen Research Fellowship}

\section{Introduction}
Let $\Delta\colon G \to \GL(n, K)$ be a representation, where $G$ is an arbitrary group (possibly infinite) 
and $K$ is an arbitrary field.
Denote by $\chi = \chi_\Delta\colon G \to K\colon g \mapsto \tr(\Delta(g))$ its character.
In this paper, we assume knowledge of $\chi$ only, 
and study which properties of $\Delta$ can be inferred.
We will restrict to absolutely irreducible representations; in this case,
$\Delta$ is uniquely determined by~$\chi$, up to equivalence.

The results are motivated by Aschbacher's classification of maximal subgroups of $\GL(n, F)$, where
$F$ is a finite field \cite{aschbacher}. According to this classification, every subgroup of $\GL(n, F)$ belongs
to one of nine classes, often denoted $\mcC_1$ up to $\mcC_9$.
For example, a subgroup of $\GL(n, F)$ belongs to class $\mcC_2$ if acts imprimitively
on~$F^{n \times 1}$; it belongs to class $\mcC_5$ if it is definable modulo scalars
over a subfield of $F$.
We give criteria on $\chi$ to decide whether the image $H$ of $\Delta$
belongs to one of the classes $\mcC_2$, $\mcC_5$, or $\mcC_8$.
While motivated by Aschbacher's classification, most of the results are also valid for arbitrary fields.
In most cases, the criterion is of the form that $\chi$ has a non-trivial stabilizer under 
certain actions.

Most of our results are generalizations of results in \cite{l2q},
where, Plesken and Fabia{\'n}ska describe an $\Ln_2$-quotient algorithm.
This algorithm, its generalization~\cite{l2qgen}, and the $\Ln_3$-$\Uni_3$-quotient
algorithm \cite{thesis}, provide examples where only the character of a representation
is known, but not the actual representation.
The algorithms take as input a finitely presented group $G$ on two generators and compute all quotients
of $G$ which are isomorphic to $\PSL(2, q)$, $\PSL(3, q)$, or $\PSU(3, q)$.
Instead of constructing the possible representations into $\PSL(2, q)$ or $\PSL(3, q)$,
they construct all possible characters.
One advantage of this approach is that the representations recovered from the characters
are pairwise non-equivalent, whereas that would not necessarily be the case if 
we constructed the representations directly.
Another advantage is that the minimal splitting field of a representation over finite fields is
the field generated by the character values, and the latter can be easily determined from the character.
The characters are constructed 
by translating the group relations into arithmetic conditions for the possible
character values. This yields characters $\chi\colon F_2 \to R$, 
where $F_2$ is the free group on two generators and $R$ is a finitely generated commutative ring.
Taking quotients of $R$ yields characters $\chi_q\colon F_2 \to \F_q$,
and every such character corresponds to a representation $\Delta_q\colon F_2 \to \SL(n, q)$ which
induces a homomorphism $\delta_q\colon G \to \PSL(n, q)$ (with $n = 2$ or $n = 3$,
depending on the algorithm).
To decide whether $\delta_q$ is surjective, we must decide whether the image of $\Delta_q$ lies in one of the
Aschbacher classes.
This can be done using the criteria on~$\chi_q$ described in this paper.
These criteria are independent of the characteristic of the field, so instead of applying the criteria
to every $\chi_q$, they can be applied to $\chi$, thus deciding the membership
simultaneously for every quotient of~$R$.
Using this approach, the algorithms can handle all prime powers~$q$ at once, and determine
all possible values of~$q$ automatically.
The decision whether $\delta_q$ maps onto $\PSU(3,q)$ is similar.

The problem of determining the Aschbacher class of a given matrix group
$G \leq \GL(n, q)$ is central in the matrix group recognition project;
see \cite{niemeyer_praeger,holt_c2,glasby_c5,carlson_c5} 
for algorithms dealing with Aschbacher classes $\mcC_2$, $\mcC_3$, $\mcC_5$, 
and $\mcC_8$.
The results in this paper do not aim to replace any of these algorithms;
while they could be applied to matrix group recognition, the resulting runtime
would certainly be worse than that of the existing algorithms.
The algorithmic value of our results is rather that they can be applied to 
characters, without the knowledge of the full representation, and that they
work for arbitrary fields.
Furthermore, our results have a purely theoretical value, linking Aschbacher 
classes to characters admitting certain stabilizers 
(see~Theorem~\ref{T:stabilizer}).

The results in this paper are based on the following proposition,
which is a special case of \cite[Th\'eor\`eme~1]{carayol} and \cite[Theorem~6.12]{nakamoto}.
We include a short proof for the convenience of the reader.

\begin{prop}
\label{P:absolutely_equivalent}
Let $G$ be a group, $K$ a field, and let $\Delta_i \colon G \to \GL(n, K)$ be representations for $i = 1, 2$ with $\chi_{\Delta_1} = \chi_{\Delta_2}$.
Assume that $\Delta_1$ is absolutely irreducible.
Then $\Delta_1$ and $\Delta_2$ are equivalent; in particular, $\Delta_2$ is absolutely irreducible.
\end{prop}
\begin{proof}
Denote the representations of the group algebra whose restriction to $G$ is $\Delta_i$ again by~$\Delta_i$.
Let $\chi = \chi_{\Delta_1} = \chi_{\Delta_2}$,
and let $\rad(\chi)$ be the radical of the trace bilinear form 
$KG \times KG \to K\colon (x,y) \mapsto \chi(xy)$.
Then $\ker(\Delta_i) \subseteq \rad(\chi)$, 
so $\varphi_i\colon \Delta_i(KG) \to KG/\rad(\chi)\colon \Delta_i(x) \mapsto x + \rad(\chi)$ are epimorphisms.
But $\Delta_1(KG) = K^{n \times n}$ is simple, so $\varphi_1$ is invertible.
Comparing dimensions we see that $\varphi_2\circ\varphi_1^{-1}\colon K^{n \times n} \to K^{n \times n}\colon \Delta_1(x) \mapsto \Delta_2(x)$
is an automorphism, which
must be inner by the Skolem-Noether Theorem (\cite[Theorem~4.9]{jacobson2}).
\end{proof}

\section{Actions on characters}
\label{S:actions}

\begin{definition}
\label{D:actions} 
Let $\chi$ be the character of a representation $\Delta\colon G \to \GL(n, K)$.
\begin{enumerate}
    \item
    For $\alpha \in \Gal(K)$ define ${}^\alpha\chi$ by $({}^\alpha\chi)(g) := \alpha(\chi(g))$ for $g \in G$.

    \item
    For $\sigma \in \Hom(G, K^*)$ define ${}^\sigma\chi$ by $({}^\sigma\chi)(g) := \sigma(g)\chi(g)$ for $g \in G$.

    \item
    Let $\Cyc_2 = \langle \gamma \rangle$ be a cyclic group of order~2 generated by $\gamma$.
    Define ${}^\gamma\chi$ by $({}^\gamma\chi)(g) := \chi(g^{-1})$ for $g \in G$.
\end{enumerate}
\end{definition}

Clearly ${}^\alpha\chi$, ${}^\sigma\chi$, and ${}^\gamma\chi$ are characters of the representation ${}^\alpha\Delta\colon G \to \GL(n, K)\colon g \mapsto \alpha(\Delta(g))$,
${}^\sigma\Delta\colon G \to \GL(n, K)\colon g \mapsto \sigma(g)\Delta(g)$, and the contragredient representation $\Delta^{-\trs}\colon G \to \GL(n, K)\colon g \mapsto (\Delta(g)^{-1})^{\trs}$, respectively,
so we get actions of $\Gal(K)$, $\Hom(G, K^*)$, and $\langle \gamma \rangle$ on the set of all characters.

Furthermore, for $\alpha \in \Gal(K)$ and $\sigma \in \Hom(G, K^*)$ define ${}^\alpha\sigma \in \Hom(G, K^*)$ by $({}^\alpha\sigma)(g) := \alpha(\sigma(g))$ for all $g \in G$,
and ${}^\gamma\sigma \in \Hom(G, K^*)$ by $({}^\gamma\sigma)(g) := \sigma(g^{-1})$ for all $g \in G$.
This defines a semi-direct product $\Omega(G, K) := (\langle \gamma \rangle \times \Gal(K)) \ltimes \Hom(G, K^*)$,
and it is easy to check that the three actions of Definition~\ref{D:actions} yield an action
of $\Omega(G, K)$ on the set of characters of representations $G \to \GL(n, K)$.

\section{Aschbacher class $\mcC_2$}
\label{S:C2}

In this section, $K$ is an arbitrary field, unless specified otherwise.
Let $\Delta\colon G \to \GL(n, K)$ be an irreducible representation;
denote by $V := K^{n \times 1}$ the induced $KG$-module.
Let $N \tri G$ have finite index; denote by $V_N$ the restricted $KN$-module.
Let $W_1, \dotsc, W_k \leq V_N$ be representatives of the isomorphism classes of simple $KN$-submodules of $V_N$,
and let $V_i$ be the $W_i$-homogeneous component of $V_N$, that is, the sum of all submodules of $V_N$
isomorphic to $W_i$.
By Clifford's Theorem (see for example \cite[Theorem~3.6.2]{lux}), $V_i \cong \bigoplus_{j = 1}^e W_i$ with $e$ independent of $V_i$,
and $V = \bigoplus_{i = 1}^k V_i$; furthermore, $G$ acts transitively on the $V_i$.

If $V$ permits a direct sum decomposition $V = V_1 \oplus \dotsb \oplus V_k$ as vector spaces such that $G$ permutes the $V_i$ transitively,
then $\Delta$ is \textit{imprimitive with blocks} $V_1, \dotsc, V_k$, or $G$ \textit{acts imprimitively on the blocks} $V_1, \dotsc, V_k$.
Define a homomorphism $\psi\colon G \to \Sym_k$ by $gV_i = V_{\psi(g)(i)}$ for $g \in G$ and $i \in \{1, \dotsc, k\}$;
then $\Delta$ \textit{is imprimitive with block action~$\psi$}. The aim of this section is, given $\psi \colon G \to \Sym_k$, 
to provide criteria on $\chi$ to decide whether $\Delta$ is imprimitive with block action~$\psi$.
We start out with general $k$ and $\psi$, but later restrict to the special case that $k$ is prime and $\im\psi$ is solvable.

We have the following necessary condition on~$\chi$.

\begin{lemma}
\label{L:fixedpointfree}
Let $\Delta\colon G \to \GL(n, K)$ be an absolutely irreducible imprimitive representation with blocks $V_1, \dotsc, V_k$,
and let $\chi$ be the character of $\Delta$.
Let $\psi\colon G \to \Sym_k$ be a homomorphism such that $G$ acts on the blocks via~$\psi$.
Then $\chi(g) = 0$ for all $g \in G$ with $\psi(g)$ fixed-point free.
\end{lemma}
\begin{proof}
Let $T \leq G$ be the stabilizer of $V_1$. Then $V_1$ is a $KT$-module, and $V \cong KG \otimes_{KT} V_1$ by
Clifford's Theorem. 
Let $\Gamma \colon T \to \GL(n/k, K)$ be a representation of $V_1$, and let $P \colon G \to \GL(k, K)$
be the permutation representation corresponding to~$\psi$.
Let $h_1, \dotsc, h_k$ be representatives of the cosets of $T$ in $G$.
After conjugation, we may assume that the images of $\Delta$ are Kronecker products: 
that is, $\Delta(g) = P(h_i)\otimes \Gamma(t)$, where $g = h_it$ with $t \in T$.
In particular, $\tr(\Delta(g)) = 0$ if $\psi(g) = \psi(h_i)$ is fixed-point free.
\end{proof}

The converse is not true in general, that is,
if $\chi(g) = 0$ for all $g \in G$ with $\psi(g)$ fixed-point free,
then there does not necessarily exist a direct sum decomposition $V = V_1 \oplus \dotsb \oplus V_k$
such that $G$ acts on the blocks via~$\psi$.
For example, let $G = \Alt_5$, and let $\Delta\colon G \to \GL(5, \C)$ 
be the unique absolutely irreducible representation of degree~$5$; let $\psi\colon G \to \Sym_5$ be the embedding. 
Then $\chi(g) = 0$ for every $5$-cycle $g$, but $\Delta$ is primitive.
Our aim is to give some conditions which imply the converse.

The following is a partial converse of Lemma~\ref{L:fixedpointfree}.

\begin{lemma}
\label{L:impred}
Let $\Delta\colon G \to \GL(n, K)$ be an absolutely irreducible representation with character~$\chi$,
and let $\psi\colon G \to \Sym_k$ be a homomorphism with kernel~$N$.
If there exists $g \in G$ such that $\chi(gx) = 0$ for all $x \in N$,
then $\Delta_{|N}$ is not absolutely irreducible.
\end{lemma}
\begin{proof}
Suppose $\Delta_{|N}$ is absolutely irreducible, so $\Delta(N)$ contains a basis of~$K^{n \times n}$.
Let $S\colon K^{n \times n} \times K^{n \times n} \to K$ be the trace bilinear form.
Then $S(\Delta(g), \Delta(x)) = \chi(gx) = 0$ for all $x \in N$,
and since $S$ is non-degenerate, this implies $\Delta(g) = 0$, which is impossible.
\end{proof}

We now restrict to the case $k = p$, a prime.

\begin{thm}
\label{T:imprimitive}
Let $K$ be algebraically closed, and let $\Delta\colon G \to \GL(n, K)$ be irreducible with character~$\chi$.
Let $p$ be a prime with $(n, p-1) = 1$, and let $\psi\colon G \to \Sym_p$ be a homomorphism
such that the image is transitive and solvable.
Then $\Delta$ is imprimitive with block action~$\psi$ if and only if $\chi(g) = 0$ for all $g \in G$ with $\psi(g)$ fixed-point free.
\end{thm}
\begin{proof}
By Lemma~\ref{L:fixedpointfree} the condition is necessary. We prove that it is sufficient.
Let $V := K^{n \times 1}$ be the $KG$-module induced by $\Delta$, and let $N := \ker\psi$.
By the O'Nan-Scott Theorem (see for example~\cite[Theorem~4.1A]{dixon}),
a transitive subgroup of $\Sym_p$ is either almost simple or isomorphic to a subgroup of $\AGL(1, p)$,
where $\AGL(1, p)$ denotes the one-dimensional affine group over~$\F_p$, acting on~$\F_p$.
We identify $\AGL(1, p)$ with $\F_p \rtimes \F_p^*$. The transitive subgroups of $\AGL(1, p)$
are conjugate to $\F_p \rtimes H$ for $H \leq \F_p^*$, so $G/N \cong \F_p \rtimes H$ for some $H \leq \F_p^*$.
Let $C$ be the preimage of $\F_p$ under~$\psi$. Then $N \tri C \tri G$, and $C/N$ is cyclic of order~$p$.
Note that $V_C$ is irreducible by Clifford's Theorem, since $G/C$ is cyclic and $|G/C|$ is coprime to~$n$.
On the other hand, $V_N$ is not irreducible by Lemma~\ref{L:impred},
thus $V_N = W_1 \oplus \dotsb \oplus V_\ell$. Since $C/N$ is cyclic, the $W_i$ are not all pairwise isomorphic,
so $V_N = V_1 \oplus \dotsb \oplus V_k$ with $k > 1$ in the notation of Clifford's Theorem above.
But $C/N$ is cyclic of order $p$ and acts transitively on the $V_i$, hence $k = p$.
Now $G/N$ acts on the blocks $V_1, \dotsc, V_p$. Since there is only one conjugacy class of subgroups of $G/N$ of index~$p$,
the action of $G/N$ on the blocks is isomorphic to the action induced by~$\psi$.
\end{proof}

The last proof used the fact that $K$ is algebraically closed, and in fact the statement is no longer
true for arbitrary fields; it already fails for a cyclic action.
For example, let $G := \F_{q^2}^* \rtimes \Aut(\F_{q^2}/\F_q)$, and let $\Delta\colon G \to \GL(2, q)$ be an embedding.
Let $\psi \colon G \to \Sym_2$ be the projection onto $\Aut(\F_{q^2}/\F_q)$, so $N := \F_{q^2}^* \tri G$. 
Then $\Delta_{|N}$ is irreducible, but not absolutely irreducible; $\F_{q^2}$ is the smallest splitting field for $\Delta_{|N}$.
Thus $\Delta$ is imprimitive only over the field~$\F_{q^2}$, not over the field~$\F_q$.
However, we can prove a variant of Theorem~\ref{T:imprimitive} over arbitrary fields
with some restrictions on the representation.

\begin{thm}
Let $K$ be an arbitrary field.
Let $p$ be prime and let $\Delta\colon G \to \GL(p, K)$ be an absolutely irreducible representation with character~$\chi$.
Let $\psi\colon G \to \Sym_p$ be a homomorphism such that the image is transitive and solvable, but not cyclic.
Then $\Delta$ is imprimitive with block action~$\psi$ if and only if $\chi(g) = 0$ for all $g \in G$ with $\psi(g)$ fixed-point free.
\end{thm}
\begin{proof}
We use the notation of the proof of Theorem~\ref{T:imprimitive}.
There exists an extension field $L/K$ such that $\Delta$ is imprimitive over~$L$,
so $L \otimes_K V_N = V_1 \oplus \dotsb \oplus V_p$ for one-dimensional
$LN$-modules $V_1, \dotsc, V_p$.
We show that we can choose $L = K$.
Let $T := \Stab_G(V_1)$, and let $\Gamma\colon T \to L^* = \GL(1, L)$ be the induced representation.
By Clifford's Theorem, $L \otimes_K V \cong (V_1)_T^G = V_1 \otimes_{LT}LG$, so $\chi = \Gamma^G$.
Note that $\AGL(1, p)$ is a Frobenius group, so the intersection of two distinct stabilizers is trivial.
Thus if $t \in T \setminus N$, then $gtg^{-1} \not\in T$ for all $g \in G \setminus T$.
The formula for induced characters shows $\Gamma(t) = \chi(t) \in K^*$ for all $t \in T \setminus N$.
Since $T$ has index $p$ in $G$ and $G/N$ is not cyclic, $T \neq N$.
Fix $t \in T \setminus N$, and let $n \in N$.
Then $nt \in T  \setminus  N$, so $\Gamma(n)\Gamma(t) = \Gamma(nt) = \chi(nt) \in K^*$, hence $\Gamma(n) \in K^*$.
Thus $\Gamma$ is realized over $K$.
Let $V_1'$ be the $KT$-module induced by $\Gamma$; let $\Delta'\colon G \to \GL(p, K)$ be the representation induced by $(V_1')_T^G$ and $\chi'$ the character of $\Delta'$.
Then $\Delta'$ is imprimitive, and $\chi' = \Gamma^G = \chi$, so $\Delta'$ is equivalent to $\Delta$.
\end{proof}

It would be very interesting to find general criteria which work for actions of non-prime degree
or with non-solvable image. However, it seems that other techniques are needed for this.
The problem with the current approach is that we do not have control over the action on the blocks;
in the proof of Theorem~\ref{T:imprimitive} we use that $V_N$ decomposes into $p$ blocks, and there
is only one transitive action of $\AGL(1, p)$ on a $p$-element set, but this is no longer true for
arbitrary groups.

\section{Aschbacher class $\mcC_5$}
\label{S:C5}

The results in this section are only valid for finite fields, so we assume that $K$ is finite.
An absolutely irreducible subgroup $H \leq \GL(n, K)$ 
is in Aschbacher class $\mcC_5$ if it is conjugate to a subgroup
of $\GL(n, F)K^*$ for some subfield $F < K$, where we identify $K^*$ with
the scalar matrices in $\GL(n, K)$.
Let $\Nor_{K/F}\colon K^*\to F^*$ be the norm; then $\ker(\Nor_{K/F})$ is
the set of elements of~$K$ which have norm~$1$ over~$F$.

\begin{thm}
\label{T:PGL}
Let $K/F$ be an extension of finite fields, 
and let $\alpha$ be a generator of the Galois group of $K/F$.
Let $\Delta\colon G \to \GL(n, K)$ be an absolutely irreducible representation with character~$\chi$.
The image of $\Delta$ is conjugate to a subgroup of $\GL(n, F)K^*$ if and only if ${}^\alpha\chi = {}^\sigma\chi$
for some $\sigma \in \Hom(G, \ker(\Nor_{K/F}))$.
\end{thm}
\begin{proof}
Assume first that the image is conjugate to a subgroup of $\GL(n, F)K^*$.
We may assume that it is in fact a subgroup of $\GL(n, F)K^*$, since we are 
interested only in the traces.
For every $g \in G$ there exist $\lambda_g \in K^*$ and $X_g \in \GL(n, F)$ with $\Delta(g) = \lambda_g X_g$.
Set $\sigma(g) := \alpha(\lambda_g)\lambda_g^{-1}$. This defines a homomorphism $\sigma\colon G \to \ker(\Nor_{K/F})$
with the desired properties.

Now assume conversely ${}^\alpha\chi = {}^\sigma\chi$.
We show first that we can assume ${}^\alpha\Delta = {}^\sigma\Delta$,
by adapting an argument of \cite{glasby}.
By Proposition~\ref{P:absolutely_equivalent}, ${}^\alpha\Delta$ and ${}^\sigma\Delta$
are equivalent, so $y({}^\sigma\Delta)y^{-1} = {}^\alpha\Delta$ for some $y \in \GL(n, K)$.  
Let $g \in G$. Then
\begin{align*}
    \Delta(g) = \alpha^{\ell-1}(y\sigma(g)\Delta(g)y^{-1}) 
    = \Nor_{K/F}(\sigma(g))\alpha^{\ell-1}(y)\dotsb \alpha(y)y \Delta(g) y^{-1}\alpha(y)^{-1}\dotsb\alpha^{\ell-1}(y)^{-1},
\end{align*}
where $\ell = |\alpha|$.
Since $g$ is arbitrary and $\Delta$ is absolutely irreducible, Schur's Lemma
yields $\alpha^{\ell-1}(y)\dotsb\alpha(y)y = \lambda I_n$ for some $\lambda \in K^*$.
Applying $\alpha$ to this equation and conjugating with~$y^{-1}$, we see that
$\lambda$ is fixed by $\alpha$, hence $\lambda \in F$.
Choose $\eta \in K^*$ with $\Nor_{K/F}(\eta) = \lambda$; 
replacing $y$ by $\eta y \in \GL(n, K)$ 
we may assume that $\alpha^{\ell-1}(y)\dotsb\alpha(y)y = I_n$.
Hilbert's Theorem~90 for matrices applies (see \cite[Proposition~1.3]{glasby}),
so there exists $z \in \GL(n, K)$ with $y = \alpha(z)^{-1}z$.
Since ${}^\alpha({}^z\Delta) = {}^\sigma({}^z\Delta)$, for the
rest of the proof we may assume ${}^\alpha\Delta = {}^\sigma\Delta$ and we will
show that the image of $\Delta$ is a subgroup of $\GL(n, F)K^*$.

Since $\sigma(g)$ has norm~$1$, there exists $\lambda_g \in K^*$ with $\sigma(g) = \alpha(\lambda_g)\lambda_g^{-1}$
by Hilbert's Theorem~90. Set $X_g := \lambda_g^{-1}\Delta(g)$.
Then $\alpha(X_g) = \alpha(\lambda_g)^{-1}\sigma(g)\Delta(g) = X_g$,
so $X_g \in \GL(n, F)$.
\end{proof}

\section{Aschbacher class $\mcC_8$}
\label{S:C8}

We study absolutely irreducible representations preserving a symmetric, alternating or Hermitian form.
If the field $K$ is finite, then these representations lie in Aschbacher class~$\mcC_8$,
but the results are valid for arbitrary fields~$K$.

\begin{prop}
\label{P:orthogonal}
Let $G$ be a group, $K$ a field of characteristic $\neq 2$, and $\Delta\colon G \to \GL(n, K)$
an absolutely irreducible representation with character $\chi$.
Then $\Delta$ fixes a symmetric or alternating form modulo scalars if and only if ${}^{\sigma\gamma}\chi = \chi$
for some $\sigma \in \Hom(G, K^*)$.
The restriction ``modulo scalars'' can be removed if and only if we can choose $\sigma = 1$.
\end{prop}
\begin{proof}
Assume first that $\Delta$ preserves a symmetric or alternating form modulo scalars.
That is, there exists a symmetric or skew-symmetric matrix $y \in K^{n \times n}$
such that $\Delta(g)^{\trs}y\Delta(g) = \sigma(g)y$ for all $g \in G$, where $\sigma(g) \in K^*$.
It is easy to verify that $\sigma\colon G \to K^*\colon g \mapsto \sigma(g)$ defines a homomorphism.
Let $z \in \ol{K}^{n \times n}$ such that $z^{\trs}z = y$, where $\ol{K}$ is an algebraic closure of~$K$.
Then $z\Delta(g)z^{-1} = \sigma(g)(z\Delta(g)z^{-1})^{-\trs}$,
so $\chi(g) = \tr(z\Delta(g)z^{-1}) = \sigma(g)\tr((z\Delta(g)z^{-1})^{-\trs}) = {}^{\sigma\gamma}\chi(g)$.

Now assume that ${}^{\sigma\gamma}\chi = {}^\chi$.
The representations $\Delta$ and ${}^\sigma\Delta^{-{\trs}} = (g \mapsto \sigma(g)\Delta(g)^{-{\trs}})$
are absolutely irreducible with the same traces, so by Proposition~\ref{P:absolutely_equivalent} they are equivalent.
Let $y \in \GL(n, K)$ such that $y\Delta(g)y^{-1} = \sigma(g)\Delta(g)^{-\trs}$
for all $g \in G$.
Then
$\Delta(g) = (\Delta(g)^{-\trs})^{-\trs} = y^{-\trs}y\Delta(g)y^{-1}y^{\trs}$
for all $g \in G$, so $y^{-\trs}y$ lies in the centralizer of $\Delta$ by Schur's Lemma.
Hence $y^{-\trs}y = \lambda I_n$ for some $\lambda \in K$.
But $y^{\trs} = \lambda y = \lambda^2y^{\trs}$, 
so either $\lambda = 1$, in which case $y$ is symmetric,
or $\lambda = -1$, in which case $y$ is skew-symmetric.
\end{proof}

\begin{prop}
\label{P:unitary}
Let $G$ be a group and $K$ a field which has an automorphism $\alpha$ of order~$2$;
let $\Delta\colon G \to \GL(n, K)$ be an absolutely irreducible representation 
with character $\chi$.
Then $\Delta$ fixes an $\alpha$-Hermitian form modulo scalars if and only if ${}^{\sigma\alpha\gamma}\chi = \chi$
for some $\sigma \in \Hom(G, K^*)$.
The restriction ``modulo scalars'' can be removed if and only if we can choose $\sigma = 1$.
\end{prop}
\begin{proof}
Again, only the `if' part is non-trivial.
As in Proposition~\ref{P:orthogonal}, there exists $y \in \GL(n, K)$ with $y\Delta(g)y^{-1} = \sigma(g)\alpha(\Delta(g)^{-\trs})$
for all $g \in G$, and $y = \lambda \alpha(y)^{\trs}$ for some $\lambda \in K$.
Applying $\alpha$ to this last equation and transposing gives $\alpha(\lambda)^{-1} = \lambda$,
so $\lambda$ has norm~1 over the fixed field.
By Hilbert's Theorem~90, there exists $\mu \in K$ with
$\lambda = \alpha(\mu)/\mu$,
and replacing $y$ by $\mu y$ we may assume that $y$ is Hermitian.
\end{proof}

\section{Actions on characters revisited}

Let $\Omega(G, K) = (\langle \gamma \rangle \times \Gal(K)) \ltimes \Hom(G, K^*)$ be the group defined in Section~\ref{S:actions}, acting on the set of characters
of representations $G \to \GL(n, K)$.
The results in this paper show that characters with non-trivial stabilizers usually come from special types
of representations. The most precise statement can be made if $K$ is finite.
Let $\pi_i$ be the projection onto the $i$th factor of $\Omega(G, K)$, for $i = 1,2,3$.

\begin{thm}
\label{T:stabilizer}
Let $K$ be a finite field and $\chi$ the character of an absolutely irreducible representation $\Delta\colon G \to \GL(n, K)$.
Assume that $\chi$ has a non-trivial stabilizer; let $\rho \in \Stab(\chi)$ be an element of prime order.
\begin{enumerate}
    \item
    If $\pi_1(\rho) = 1$ and $\pi_2(\rho) = 1$, then $\Delta$ is imprimitive over an extension field of $K$, with cyclic block action.

    \item
    If $\pi_1(\rho) = 1$ and $\pi_2(\rho) \neq 1$, then $\Delta$ is realizable modulo scalars over a proper subfield of~$K$.

    \item
    If $\pi_1(\rho) \neq 1$ and $\pi_2(\rho) = 1$, and if $\car(K) \neq 2$, then $\Delta$ fixes an alternating or symmetric form modulo scalars.

    \item
    If $\pi_1(\rho) \neq 1$ and $\pi_2(\rho) \neq 1$, then $\Delta$ fixes a Hermitian form modulo scalars.
\end{enumerate}
The restriction ``modulo scalars'' in the last three statements can be removed if $\pi_3(\rho) = 1$.
\end{thm}
\begin{proof}
Let $p = |\rho|$.

Assume first $\pi_1(\rho) = 1$ and $\pi_2(\rho) = 1$, so $\rho = \sigma \in \Hom(G, K^*)$.
Then $\sigma \in \Hom(G, \langle \zeta \rangle)$, where $\zeta \in K$ is a primitive $p$th root of unity.
Identifying $\zeta$ with a $p$-cycle in $\Sym_p$ we may regard $\sigma$ as a homomorphism $\psi\colon G \to \Sym_p$.
Since $\chi(g) = {}^\sigma\chi(g) = \sigma(g)\chi(g)$ for all $g \in G$, we see $\chi(g) = 0$ whenever $\psi(g) = \sigma(g) \neq 1$,
that is, $\psi(g)$ is fixed-point free. Thus the first claim follows by Theorem~\ref{T:imprimitive}.

Now assume $\pi_1(\rho) = 1$ and $\pi_2(\rho) \neq 1$. If $\pi_3(\rho) = 1$, then $\rho = \alpha \in \Gal(K)$,
so all character values lie in the fixed field $F$ of~$\alpha$.
By Wedderburn's Theorem, a representation over a finite field is realizable over its character field,
so the claim follows for $\pi_3(\rho) = 1$.
If $\pi_3(\rho) \neq 1$, then $\rho = (\alpha, \sigma)$ for $\alpha \in \Gal(K)$ and $\sigma \in \Hom(G, K^*)$.
But $(1,1) = (\alpha, \sigma)^p = (\alpha^p, \sigma\cdot {}^\alpha\sigma\dotsb {}^{\alpha^{p-1}}\sigma)$,
so $\sigma \in \Hom(G, \ker(\Nor_{K/F}))$, where $F$ is the fixed field of $\alpha$.
The second claim now follows by Theorem~\ref{T:PGL}.

The last two claims are reformulations of Propositions~\ref{P:orthogonal} and~\ref{P:unitary}.
\end{proof}

\section{Acknowledgments}

This paper is part of my PhD thesis, written under the supervision of Wilhelm Plesken.
I am very grateful for his constant advice and support. 
I am also very grateful to Ben Martin and Eamonn O'Brien for helpful comments on early versions of this paper.
\bibliography{paper}

\noindent
Department of Mathematics \\
The University of Auckland \\
Private Bag 92019 \\
Auckland \\
New Zealand \\
E-mail address: \texttt{s.jambor@auckland.ac.nz}
\end{document}